\documentclass[12pt,a4paper]{amsart}

\usepackage[colorlinks,
            linkcolor=reference,
            citecolor=citation,
            urlcolor=e-mail,
            ]{hyperref}
\usepackage{graphics,epic}
\usepackage{subfigure}
\usepackage{amsmath,amssymb, amsthm}
\usepackage[all,2cell]{xy}
\usepackage{times}
\usepackage{tikz}
\usetikzlibrary{arrows}
\usepackage{blkarray}

\usepackage{color}
\definecolor{todo}{rgb}{1,0,0}
\definecolor{conditional}{rgb}{0,1,0}
\definecolor{e-mail}{rgb}{0,.40,.80}
\definecolor{reference}{rgb}{.20,.60,.22}
\definecolor{mrnumber}{rgb}{.80,.40,0}
\definecolor{citation}{rgb}{0,.40,.80}

\theoremstyle{plain}
\newtheorem{theorem}{Theorem}[section]
\newtheorem*{theorem*}{Theorem}

\newtheorem{lemma}[theorem]{Lemma}
\newtheorem{proposition}[theorem]{Proposition}

\newtheorem*{conjecture*}{Conjecture}

\theoremstyle{definition}

\newtheorem{example}[theorem]{Example}
\newtheorem{remark}[theorem]{Remark}
\newtheorem{definition}[theorem]{Definition}

\DeclareMathOperator{\mut}{Mut}

\renewcommand{\emptyset}{\varnothing}

\newcommand{\opname}[1]{\operatorname{\mathsf{#1}}}

\renewcommand{\mod}{\opname{mod}\nolimits}

\newcommand{\add}{\opname{add}\nolimits}

\newcommand{\der}{\cd}

\newcommand{\X}{\mathbb{X}}

\newcommand{\Z}{\mathbb{Z}}
\newcommand{\N}{\mathbb{N}}

\newcommand{\D}{\mathbb{D}}

\renewcommand{\P}{\mathbb{P}}

\newcommand{\sgn}{\opname{sgn}}

\newcommand{\Hom}{\opname{Hom}}

\newcommand{\Ext}{\opname{Ext}}

\newcommand{\coh}{\opname{coh}}
\newcommand{\End}{\opname{End}}

\newcommand{\vect}{\opname{vect}}

\newcommand{\cd}{{\mathcal D}}

\newcommand{\mt}{\mathcal{T}}

\renewcommand{\hat}[1]{\widehat{#1}}

\begin{document}

\title[Existence/non-existence of MGS]{On maximal green sequence for quivers arising from weighted projective lines}\thanks{Partially supported by the National Natural Science Foundation of China (Grant No. 11971326, 12071315)}
\author{Changjian Fu}
\address{Changjian Fu\\Department of Mathematics\\SiChuan University\\610064 Chengdu\\PR China}
\email{changjianfu@scu.edu.cn}
\author{Shengfei Geng}
\address{Shengfei Geng\\Department of Mathematics\\SiChuan University\\610064 Chengdu\\PR China}
\email{genshengfei@scu.edu.cn}
\maketitle

\begin{abstract}
We investigate the existence and non-existence of maximal green sequences for quivers arising from weighted projective lines. Let $Q$ be the Gabriel quiver of the endomorphism algebra of a basic cluster-tilting object in the cluster category $\mathcal{C}_\X$ of a weighted projective line $\X$. It is proved that there exists a quiver $Q'$ in the mutation equivalence class $\mut(Q)$ of $Q$ such that $Q'$ admits a maximal green sequence. Furthermore, there is a quiver in $\mut(Q)$ which does not admit a maximal green sequence if and only if $\X$ is of wild type.
\end{abstract}

\section{Introduction}
Maximal green sequences were introduced by Keller \cite{Keller11} in the study of refined Donaldson-Thomas invariants for quivers and implicitly by Gaiotto et al. in \cite{GMN}. They are certain sequences of quiver mutations satisfying a certain combinatorial condition. It is known that not all quivers have maximal green sequences, but they do exist for important classes of quivers.  We refer to the survey \cite{Keller19} for examples and recent progress. It is an open question to determine whether a given quiver admits a maximal green sequence or not.

 The existence of maximal green sequences yields quantum dilogarithm identities in the associated quantum torus and provides explicit formulas for Kontsevich-Soibelman's refined Donaldson-Thomas invariants (cf. \cite{Keller11}). It also has important applications in the theory of cluster algebras. In particular, Gross et al. \cite{GHKK} proved that the Fock-Goncharov conjecture about the existence of a canonical basis for cluster algebras holds when a cluster algebra has a quiver with a maximal green sequence and the cluster algebra is equal to its upper cluster algebra. It is also a sufficient condition for the existence of a generic basis in certain upper cluster algebras \cite{Qin}.
 
 Cluster-tilting theory of hereditary abelian categories  produces a large class  of important quivers, which we denote it by $\mathcal{Q}_{\text{ct}}$. Let $K$ be an algebraically closed field and $\mathcal{H}$ a hereditary abelian category over $K$ with tilting objects. The {\it cluster category} $\mathcal{C}(\mathcal{H})$ \cite{BMRRT} is defined as the orbit category of the bounded derived category $\der^b(\mathcal{H})$ with respect to the auto-equivalence $\tau^{-1}\circ \Sigma$, where $\tau$ is the Auslander-Reiten translation and $\Sigma$ is the suspension functor of $\der^b(\mathcal{H})$, respectively. The cluster category $\mathcal{C}(\mathcal{H})$ is a $2$-Calabi-Yau triangulated category with cluster-tilting objects (cf. \cite{Keller}). For each basic cluster-tilting object $T\in \mathcal{C}(\mathcal{H})$, we denote by $Q_T$ the Gabriel quiver of the endomorphism algebra $\End_{\mathcal{C}(\mathcal{H})}(T)$. Then  $\mathcal{Q}_{ct}$ consists of quivers which are isomorphic to $Q_T$ for some basic cluster-tilting object $T$ and hereditary abelian category $\mathcal{H}$.
 
  According to Happel's classification theorem \cite{Happel}, each connected hereditary abelian $K$-category with tilting objects is either derived equivalent to the path algebra $KQ$ of a finite acyclic quiver $Q$ or to the category $\coh\X$ of coherent sheaves over a weighted projective line in the sense of Geigle-Lenzing \cite{GeigleLenzing}. Therefore, $\mathcal{Q}_{\text{ct}}$ can be written as the union of two subclasses: $\mathcal{Q}_{\text{pa}}$ consists of quivers arising from path algebras and $\mathcal{Q}_{\text{wpl}}$ consists of quivers arising from weighted projective lines. Quivers in $\mathcal{Q}_{\text{pa}}$ and their associated cluster algebras were well-studied. It is natural to investigate the quiver in $\mathcal{Q}_{\text{wpl}}$ and their associated cluster algebras.
 The aim of this note is to study the existence and non-existence of maximal green sequences for quivers in $\mathcal{Q}_{\text{wpl}}$. Our main result is an existence and non-existence theorem (cf. Theorem \ref{t:main-result} ) for quivers arising from weighted projective lines. Surprisingly, the existence and non-existence theorem is compatible with the classification of weighted projective lines.
 
 The paper is structured as follows. In Section \ref{s:quiver-mutation}, we recall the definitions of quiver mutation and maximal green sequence. Quivers of finite mutation type are also discussed.  In Section \ref{s:weighted-projective-line}, we collect basic properties for weighted projective lines. It is proved that a quiver arising from  a weighted projective line $\X$ is of finite mutation type if and only if $\X$ is not of wild type (Proposition \ref{p:finite-mutation-wpl}). In Section \ref{s:proof-main-result}, we present the proof of the main result (Theorem \ref{t:main-result}).
\subsection*{Conventions}
Let $m\geq n$ be positive integers. For an integer matrix $B\in M_{m\times n}(\Z)$, we refer to the submatrix formed by the first $n$ rows of $B$ the {\it principal part} of $B$ and the submatrix formed by the last $m-n$ rows the {\it coefficient part}.

For any integer vectors $\alpha=[a_1,\cdots, a_n]^T, \beta=[b_1,\cdots, b_n]^T\in \Z^n$,  we denote by $\alpha\leq \beta$ if $a_i\leq b_i$ for $1\leq i\leq n$. This endows a partial order on $\Z^n$. For $b\in \mathbb{Z}$, let  $\sgn(b)$ be $1$, $0$, or $-1$, depending on whether $b$ is positive, zero, or negative.

\section{Preliminaries}\label{s:quiver-mutation}
\subsection{Quivers and mutation}
A quiver is an oriented graph, i.e., a quadruple $Q=(Q_0,Q_1,s,t)$ formed by a set of vertices $Q_0$, a set of arrows $Q_1$ and two maps $s$ and $t$ from $Q_1$ to $Q_0$ which send an arrow $\alpha$ respectively to its source $s(\alpha)$ and its target $t(\alpha)$.
An arrow whose source and target coincide is a {\it loop}; a {\it $2$-cycle} is a pair of distinct arrows $\alpha$ and $\beta$ such that $s(\alpha)=t(\beta)$ and $t(\alpha)=s(\beta)$.
 By convention, in the sequel, by a quiver we always mean a finite quiver without loops nor $2$-cycles.
 An {\it ice quiver} is a pair $(Q,F)$, where $Q$ is a quiver and $F$ is a subset of $Q_0$ called {\it frozen vertices}, such that there are no arrows between frozen vertices.
 The non-frozen vertices of $(Q,F)$ are {\it mutable vertices}. The {\it mutable part} of $(Q,F)$ is the full subquiver of $(Q,F)$ consisting of mutable vertices.
  \begin{definition}
 	Let $(Q,F)$ be an ice quiver and $k$ a mutable vertex. The \textit{mutation} $\mu_k(Q,F)$ of $(Q,F)$ at vertex $k$ is the ice quiver obtained from $(Q,F)$ as follows:
 	\begin{itemize}
 		\item for each subquiver $i\xrightarrow{\beta}k\xrightarrow{\alpha}j$, we add a new arrow $[\alpha\beta]:i\to j$;
 		\item we reverse all arrows with source or target $k$;
 		\item we remove the arrows in a maximal set of pairwise disjoint $2$-cycles and any arrows that created between frozen vertices.
 		\end{itemize}
 \end{definition}
When $F=\emptyset$, we also write  $\mu_k(Q)$ for $\mu_{k}(Q,\emptyset)$.

Let $(Q,F)$ be an ice quiver with non frozen vertices $\{1, \dots, n\}$ and frozen vertices $\{n+1,\dots, m\}$. Up to an isomorphism fixing the vertices, such an ice quiver is given by an $m\times n$ integer matrix $B(Q,F)$ whose coefficient $b_{ij}$ is the difference between the number of arrows from $j$ to $i$ and the number of arrows from $i$ to $j$.  
 In particular, the principal part of $B(Q,F)$ is skew-symmetric. Conversely, each $m\times n$ integer matrix $B$ with skew-symmetric principal part comes from an ice quiver.  Let $B(Q,F)=(b_{ij})$ be the associated matrix of $(Q,F)$. For any mutable vertex $k$, we denote by $\mu_k(B(Q,F))=(b_{ij}')$ the matrix associated to the ice quiver $\mu_k(Q,F)$, then
\[b_{ij}'=\begin{cases}-b_{ij}&\text{if $i=k$ or $j=k$;}\\
b_{ij}+\sgn(b_{ik})\max(0,b_{ik}b_{kj})&\text{else.}
\end{cases}
\]
This is the {\it matrix mutation rule} introduced by Fomin and Zelevinsky \cite{FZ}.
It is clear that $\mu_k(B(Q,F))=B(\mu_k(Q,F))$.

Mutation at a fixed vertex is an involution. Two ice quivers are {\it mutation-equivalent} if they are linked by a finite sequence of mutations. We will denote by $\mut(Q,F)$ the set of all quivers that can be obtained from $(Q,F)$ by a finite sequence of mutations. We write $\mut(Q):=\mut(Q, \emptyset)$.

\subsection{Maximal green sequence}
\begin{definition}
	Let $Q$ be a quiver. The {\it framed quiver} $\hat{Q}$ of $Q$ is the ice quiver $(\hat{Q}, Q_0^*)$ such that:
\[Q_0^*=\{i^*~|~i\in Q_0\},~ \hat{Q}_0=Q_0\sqcup Q_0^*, ~\hat{Q}_1=Q_1\sqcup \{i\to i^*~|~i\in Q_0\}.
\]
The {\it coframed quiver $\check{Q}$} is the ice quiver $(\check{Q}, Q_0^*)$ such that:
\[Q_0^*=\{i^*~|~i\in Q_0\},~ \check{Q}_0=Q_0\sqcup Q_0^*, ~\check{Q}_1=Q_1\sqcup \{i\gets i^*~|~i\in Q_0\}.
\]

\end{definition}
\begin{definition}
	Let $R\in \mut(\hat{Q},Q_0^\ast)$. A mutable vertex $k\in R_0$ is {\it green} if $\{j^*\in Q_0^*~|~\exists~  j^*\to k\in R_1\}=\emptyset$. It is {\it red} if $\{j^*\in Q_0^*~|~\exists~ j^*\gets k\in R_1\}=\emptyset$. 
\end{definition}
We have the following {\it sign-coherence} property.
\begin{theorem}\cite[Theorem 1.7]{DerksenWeymanZelevinsky}\label{t:dwz}
Every mutable vertex of $R\in \mut(\hat{Q},Q_0^\ast)$ is either green or red. 
\end{theorem}

\begin{remark}
 A non-zero integer vector $c\in \Z^n$ is {\it sign-coherent} if $c\leq 0$ or $0\leq c$.  
  Let $Q$ be a quiver with vertex set $\{1,\dots, n\}$.
For $R\in \mut(\hat{Q},Q_0^\ast)$, recall that $B(R,Q_0^\ast)$ is the associated $2n\times n$ integer matrix. Theorem~\ref{t:dwz} can be restated as follows:
each column vector of the coefficient part of $B(R,Q_0^\ast)$ is sign-coherent.
\end{remark}

\begin{definition}
	A {\it green sequence} for a quiver $Q$ is a sequence ${\bf \mathfrak{i}}=(i_1,\dots, i_l)$ of vertices  of $Q$ such that for any $1\leq k\leq l$, the vertex $i_k$ is green in $\mu_{i_{k-1}}\circ\cdots\circ \mu_{i_1}(\hat{Q},Q_0^\ast)$. The green sequence ${\bf \mathfrak{i}}$ is {\it maximal} if every mutable vertex in $\mu_{i_{l}}\circ\cdots\circ \mu_{i_1}(\hat{Q},Q_0^\ast)$ is red. We will simply denote the composition $\mu_{i_{l}}\circ\cdots\circ \mu_{i_1}$ by $\mu_{\bf \mathfrak{i}}$. A {\it green-to-red sequence} is a sequence ${\bf \mathfrak{i}}$ of vertices of $Q$ such that every mutable vertex in $\mu_{\bf \mathfrak{i}}(\hat{Q},Q_0^\ast)$ is red. 
\end{definition}

\begin{proposition}\cite[Proposition 2.10]{BDP}~\label{p:bdp}
Suppose that $Q$ admits a green-to-red sequence ${\bf \mathfrak{i}}$. Then there is a unique isomorphism $\mu_{\bf \mathfrak{i}}(\hat{Q},Q_0^\ast)\xrightarrow{\sim}\check{Q}$ fixing the frozen vertices and sending a non frozen vertex $i$ to $\sigma(i)$ for a unique permutation $\sigma$ of the vertices of $Q$.
\end{proposition}
\begin{remark}~\label{r:mgs}
By definition and Proposition~\ref{p:bdp}, it is known that a sequence ${\bf \mathfrak{i}}$ is a green-to-red sequence of $Q$ if and only if the coefficient part of the matrix $B(\mu_{\bf \mathfrak{i}}(\hat{Q},Q_0^\ast))=\mu_{\bf \mathfrak{i}}(B(\hat{Q},Q_0^\ast))$ is a permutation of $-I_n$. A sequence ${\bf \mathfrak{i}}=(i_1,\dots,i_l)$ is a maximal green sequence if and only if 
\begin{itemize}
\item the $i_k$-th column vector of the coefficient part of $B(\mu_{i_{k-1}}\circ \cdots \circ \mu_{i_1}(\hat{Q},Q_0^\ast))$ is positive for $1\leq k\leq l$;
\item the coefficient part of the matrix $B(\mu_{\bf \mathfrak{i}}(\hat{Q},Q_0^\ast))=\mu_{\bf \mathfrak{i}}(B(\hat{Q},Q_0^\ast))$ is a permutation of $-I_n$.
\end{itemize}
\end{remark}
By definition, all maximal green sequences are green-to-red sequences. There are quivers for which a maximal green sequence does not exist, but a green-to-red sequence does. Furthermore, there are quivers for which no green-to-red sequence exists. 
\begin{example}
Let  $a, b, c$ be non negative integers, denote by $Q_{a,b,c}$ the quiver with three vertices $1, 2, 3$ and $a$ arrows from $1$ to $2$, $b$ arrows from $2$ to $3$ and $c$ arrows from $3$ to $1$. It is known that $Q_{2,2,2}$ does not admit a green-to-red sequence. Furthermore, Muller~\cite[Theorem 12]{Muller} proved
 that $Q_{a,b,c}$ does not admit maximal green sequences whenever $a,b,c\geq 2$. 
\end{example}

\begin{lemma}\cite[Corollary 19]{Muller}
	If a quiver $Q$ admits a green-to-red sequence, then any quiver mutation-equivalent to $Q$ also admits a green-to-red sequence.
\end{lemma}

Muller \cite{Muller} also proved that the property of having a maximal green sequence is not invariant under mutation.
  The following is useful to show the non-existence of maximal green sequence for a given quiver.
 \begin{lemma}\cite[Theorem 9 and 17]{Muller}~\label{l:muller-existence}
 	If a quiver $Q$ admits a green-to-red sequence (resp. maximal green sequence), then any full subquiver of $Q$ also admits a green-to-red sequence (resp. maximal green sequence).
 	In particular, if $Q$ has a full subquiver $Q_{a,b,c}$ with $a,b,c\geq 2$, then $Q$ does not admit a maximal green sequence.
 \end{lemma}
 
 \begin{definition}
 Let $Q$ be a quiver and $Q', Q''$ full subquivers. We say that $Q$ is a {\it triangular extension} of $Q'$ by $Q''$ if the set of vertices of $Q$ is the disjoint union of the sets of vertices of $Q'$ and $Q''$ and there are no arrows from vertices of $Q''$ to vertices of $Q'$.
 \end{definition} 
 The following result was proved in~\cite[Theorem 4.5]{CaoLi} using Lemma~\ref{l:muller-existence}.
\begin{lemma}\cite[Theorem 4.5]{CaoLi}\label{l:CL-lemma}
If $Q$ is a triangular extension of $Q'$ and $Q''$, then $Q$ has a maximal green sequence if and only if $Q'$ and $Q''$ have maximal green sequences.
\end{lemma}

\subsection{Tropical dualities between $c$-vectors and $g$-vectors}
Let $Q$ be a quiver with vertex set $\{1,2,\dots,n\}$. Denote by $\mathbb{T}_n$ the {\it $n$-regular tree} whose edges are labeled by the numbers $1,\dots,n$ such that the $n$ edges emanating from each vertex have different labels. We write $\xymatrix{t\ar@{-}[r]^k&t'}$ to indicate that vertices $t$ and $t'$ are linked by an edge labeled by $k$.

A {\it quiver pattern} of $(\hat{Q}, Q_0^\ast)$ is an assignment of an ice quiver $R_t\in \mut(\hat{Q},Q_0^\ast)$ to each vertex $t\in \mathbb{T}_n$ such that 
\begin{enumerate}
	\item there is a vertex $t_0\in \mathbb{T}_n$ such that $R_{t_0}=(\hat{Q}, Q_0^\ast)$;
	\item if $\xymatrix{t\ar@{-}[r]^k&t',}$  then $R_{t'}=\mu_k(R_t)$.
\end{enumerate}
Clearly, a quiver pattern of $(\hat{Q}, Q_0^\ast)$ is uniquely determined by assigning $(\hat{Q}, Q_0^\ast)$ to the vertex $t_0\in \mathbb{T}_n$ and $t_0$ is called the {\it root} vertex of the quiver pattern.

We relabel the vertex $i^\ast$ as $n+i$ for each $i\in Q_0$ and fix a quiver pattern of $(\hat{Q}, Q_0^\ast)$. In particular, for each vertex $t\in \mathbb{T}_n$, we have a $2n\times n$ integer matrix $B(R_t):=(b_{ij;t})$. The coefficient part $C_t$ of $B(R_t)$ is the {\it $C$-matrix} at $t$. Its column vectors are {\it $c$-vectors}.

Let $e_1,\dots, e_n$ be the standard basis of $\mathbb{Z}^n$. For $1\leq j\leq n$, denote by $\beta_j$ the $j$th column of the principal part of $B(R_{t_0})$.
For each vertex $t\in \mathbb{T}_n$, we also assign an integer matrix $G_t:=(g_{1;t},\dots, g_{n;t})$ by the following recursion:
\begin{enumerate}
	\item for any $1\leq i\leq n$, $g_{i;t_0}=e_i$;
	\item suppose that $G_t$ is defined and let $\xymatrix{t\ar@{-}[r]^k&t'}$ be an edge of $\mathbb{T}_n$, then 
	\[g_{i;t'}=\begin{cases}g_{i;t}&i\neq k;\\
	-g_{k;t}+\sum_{j=1}^n[b_{jk;t}]_+g_{j;t}-\sum_{j=1}^n[b_{(n+j) k;t}]_+\beta_j&i=k.
	\end{cases}
	\]
\end{enumerate}
We call $G_t$ the {\it $G$-matrix} at $t$ and its column vectors are {\it $g$-vectors}. 
\begin{proposition}\cite[Theorem 1.7]{DerksenWeymanZelevinsky}
For each vertex $t\in \mathbb{T}_n$, every row vector of $G_t$ is sign-coherent.
\end{proposition}
The following is known as the tropical duality between $c$-vectors and $g$-vectors (cf. \cite{Nak,Keller12, Plamondon}).
\begin{theorem}\cite[Theorem 4.1]{Nak}\label{t:tropical-duality}
	For each vertex $t\in \mathbb{T}_n$, we have 
	\[G_t^{T}C_t=I_n.
	\]
\end{theorem}
\subsection{Finite mutation type}
A quiver $Q$ is of {\it  finite  mutation type} if $\mut(Q)$ is a finite set. 
Quivers of  finite mutation type  have been classified in \cite{FST}. Here, we only recall the following.
\begin{lemma}\label{l:finite-mutatin-type}
\begin{itemize}
\item[(1)] Every quiver with two vertices is of  finite mutation type.
\item[(2)] If $Q$ is acyclic with at least three vertices, then $Q$ is of  finite  mutation type if and only if $Q$ is of Dynkin type or extended Dynkin type.
\item[(3)] Each quiver in Figure \ref{f:quiver-tubular-type} is of  finite mutation  type.
\end{itemize}
\end{lemma}
\begin{proof}
The statement $(1)$ is obvious, $(2)$ is proved by \cite[Theorem 3.6]{BR}. For $(3)$, one can verify the finiteness by the MutationApp of Keller \cite{Kellerapp} directly (cf. also \cite[Theorem 6.1 ]{FST}).
\end{proof}

\begin{figure}
\begin{minipage}[t]{0.40\linewidth}
\centering
\[\tiny\xymatrix{D_4^{(1,1)}:&\circ\ar[r]&\circ\ar@{=>}[d] &\circ\ar[l]\\
&\circ\ar[ur]&\circ\ar[ul]\ar[ur]\ar[l]\ar[r]&\circ\ar[ul]}
\]
\end{minipage}%
\begin{minipage}[t]{0.40\linewidth}
\centering
\[\tiny\xymatrix@C=0.5cm@R=0.35cm{E_6^{(1,1)}:&&\circ\ar@{=>}[dd]&\circ\ar[r]\ar[l]&\circ\\
\circ&\circ\ar[l]\ar[ur]\\
&&\circ\ar[ul]\ar[uur]\ar[r]&\circ\ar[uul]\ar[r]&\circ
}
\]
\end{minipage}

\begin{minipage}[t]{0.45\linewidth}
\centering
\[\tiny\xymatrix@C=0.5cm@R=0.35cm{&E_7^{(1,1)}:&\circ\ar@{=>}[dd]&\circ\ar[r]\ar[l]&\circ\ar[r]&\circ\\
&\circ\ar[ur]\\
&&\circ\ar[ul]\ar[uur]\ar[r]&\circ\ar[uul]\ar[r]&\circ\ar[r]&\circ
}
\]
\end{minipage}
\begin{minipage}[t]{0.45\linewidth}
\centering
\[\tiny\xymatrix@C=0.5cm@R=0.35cm{&E_8^{(1,1)}:&\circ\ar@{=>}[dd]&\circ\ar[r]\ar[l]&\circ\\
&\circ\ar[ur]\\
&&\circ\ar[ul]\ar[uur]\ar[r]&\circ\ar[uul]\ar[r]&\circ\ar[r]&\circ\ar[r]&\circ\ar[r]&\circ
}
\]
\end{minipage}
\caption{Quivers of tubular type $D_4^{(1,1)}$, $E_6^{(1,1)}$, $E_7^{(1,1)}$ and $E_8^{(1,1)}$.}\label{f:quiver-tubular-type}
\end{figure}
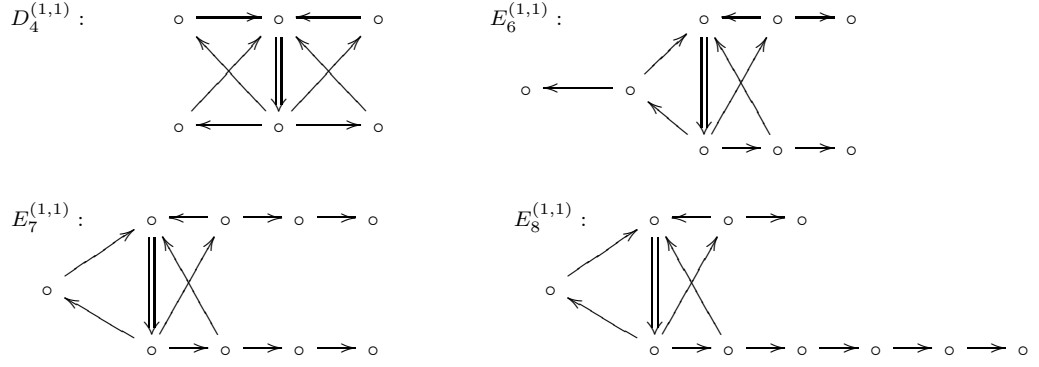

\begin{lemma}\label{l:non-finite-mutation-type}
Let $2\leq c\leq b\leq a$ be integers. If  $a\geq 3$, then the quiver $Q_{a,b,c}$ is not of  finite mutation type. 
\end{lemma}
\begin{proof}
For a non negative integer $t$, we set $Q^t:=(\mu_2\mu_1)^t(Q_{(a,b,c)})$ and a pair  of integers $(b_t,c_t)$ by the following recursion: 
\[b_0=b,c_0=c, b_t=ac_{t-1}-b_{t-1}, c_t=ab_t-c_{t-1}.
\]
We claim that $0<b_1<c_1\cdots<b_t<c_t<\cdots$. Indeed, it is straightforward to see that $0<b_1<c_1$. For $t\geq 2$, we have 
\[b_t-c_{t-1}=ac_{t-1}-b_{t-1}-c_{t-1}=(a-1)c_{t-1}-b_{t-1}\geq 2c_{t-1}-b_{t-1}>0
\]
and 
\[c_{t}-b_t=(a-1)b_t-c_{t-1}>0
\]
by induction.
As a consequence, $Q^t=Q_{(a,b_t,c_t)}$ and $Q_{(a,b,c)}$ is not of finite mutation type.
\end{proof}

\section{Quivers arising from Weighted projective lines}\label{s:weighted-projective-line}
\subsection{Weighted projective lines}
 Fix a positive integer $t\geq 2$.
A {\it weighted projective line} $\X=\X(\mathbf{p},\boldsymbol{\lambda})$ over $K$ is given by a weight sequence $\mathbf{p}=(p_1,\dots,p_t)$
of positive integers, and 
a  parameter sequence  $\boldsymbol{\lambda}=(\lambda_1\dots,\lambda_t)$  of pairwise distinct points of the projective line $\P_1(K)$. Let $\mathbb{L}$ be the rank one  abelian group generated by $\vec{x}_1,\ldots, \vec{x}_t$ with the relations
\[p_1\vec{x}_1=p_2\vec{x}_2=\cdots=p_t\vec{x}_t=:\vec{c},
\]
where the element $\vec{c}$ is called the {\it canonical element} of $\mathbb{L}$. Denote by \[\vec{\omega}:=(t-2)\vec{c}-\sum\limits_{i=1}^t\vec{x}_i\in \mathbb{L},\] which is called the {\it dualizing element} of $\mathbb{L}$. Each element $\vec{x}\in \mathbb{L}$ can be uniquely written into the {\it normal form}
\[\vec{x}=\sum_{i=1}^tl_i\vec{x}_i+l\vec{c}, ~\text{where~$0\leq l_i<p_i$ and $l\in \Z$.}
\] 
Let $\vec{x}=\sum_{i=1}^tl_i\vec{x}_i+l\vec{c}$ and $\vec{y}=\sum_{i=1}^tm_i\vec{x}_i+m\vec{c}\in \mathbb{L}$ be in normal form, denote by $\vec{x}\leq \vec{y}$ if $l_i\leq m_i$ for $i=1,\ldots, t$ and $l\leq m$. This defines a partial order on $\mathbb{L}$. It is known that each $\vec{x}\in \mathbb{L}$ satisfies exactly one of the two possibilities:
\[0\leq \vec{x}~\text{or}~\vec{x}\leq \vec{c}+\vec{\omega}.
\]
\subsection{The category $\coh\X$ of coherent sheaves}
Let  
\[S:=S({\bf p}, {\boldsymbol{\lambda}})=K[X_1,\cdots, X_t]/I
\] be the quotient of the polynomial ring $K[X_1,\cdots, X_t]$ by the ideal $I$ generated by $f_i=X_i^{p_i}-X_2^{p_2}+\lambda_iX_1^{p_1}$ for $3\leq i\leq t$.
 The algebra $S$ is $\mathbb{L}$-graded by setting $\deg X_i=\vec{x}_i$~for ~$i=1,\ldots, t$ and we have the decomposition of $S$ into $K$-subspace
\[S=\bigoplus_{\vec{x}\in \mathbb{L}}S_{\vec{x}}.
\]

The category $\coh\X$ of coherent sheaves over $\X$ is defined to be the quotient category
\[\coh\X:=\mod^{\mathbb{L}}S/\mod_0^{\mathbb{L}}S,
\]
where $\mod^{\mathbb{L}}S$ is the category of finitely generated $\mathbb{L}$-graded $S$-modules, while $\mod_0^{\mathbb{L}}S$ is the Serre subcategory of $\mathbb{L}$-graded $S$-modules of finite length. For each sheaf $E$ and $\vec{x}\in \mathbb{L}$, denote by $E(\vec{x})$ the  grading shift of $E$ with $\vec{x}$. The free module $S$ gives the structure  sheaf $\mathcal{O}$, and each line bundle is given by the grading shift $\mathcal{O}(\vec{x})$ for a unique element $\vec{x}\in \mathbb{L}$. Moreover, we have
\begin{eqnarray}~\label{e:hom}
\Hom_{\X}(\mathcal{O}(\vec{x}), \mathcal{O}(\vec{y}))=S_{\vec{y}-\vec{x}}~\text{for any ~$\vec{x},\vec{y}\in \mathbb{L}$.}
\end{eqnarray}

 In ~\cite{GeigleLenzing}, Geigle and Lenzing proved that $\coh\X$ is a connected hereditary abelian category with tilting objects and has Serre duality of the form
\begin{eqnarray}~\label{e:serre-duality}
\D\Ext^1_{\X}(E,F)=\Hom_{\X}(F, E(\vec{\omega}))
\end{eqnarray}
for all $E,F\in \coh\X$.
In particular, $\coh\X$ admits almost split sequences with the Auslander-Reiten translation $\tau$ given by the grading shift with $\vec{\omega}$. Recall that an object $T\in \coh \X$ is a {\it tilting object} if $\Ext^1_{\X}(T,T)=0$ and for $X\in \coh\X$ with $\Hom_{\X}(T,X)=0=\Ext^1_{\X}(T,X)$, we have that $X=0$.

 Denote by $\vect\X$ the full subcategory of $\coh\X$ consisting of vector bundles, i.e. torsion-free sheaves, and by $\coh_0\X$  the full subcategory consisting of sheaves of finite length, i.e. torsion sheaves.
 Each coherent sheaf is the direct sum of a vector bundle and a finite length sheaf. Each vector bundle has a finite filtration by line bundles and there is no nonzero morphism from $\coh_0\X$ to $\vect \X$.  We remark that $\coh\X$ does not contain nonzero projective objects. Denote by 
 \[{\bf p}_{\boldsymbol{\lambda}}:\mathbb{P}_1(k)\to \N,~ {\bf p}_{\boldsymbol{\lambda}}(\mu)=\begin{cases}
 p_i& \text{if $\mu=\lambda_i$ for some $i$,}\\ 1 &\text{else.}
 \end{cases}
 \]
 the weight function associated with $\X$. 
\begin{proposition}~\cite[Proposition 2.5]{GeigleLenzing}~\label{p:structure-finite}
 	The category $\coh_0\X$ is an exact abelian, uniserial subcategory of $\coh\X$ which is stable under Auslander-Reiten translation. The components of the Auslander-Reiten quiver of $\coh_0\X$ form a family of pairwise orthogonal standard tubes $(\mt_{\mu})_{\mu\in\mathbb{P}_1(k)}$, where each tube $\mt_{\mu}$ has  rank ${\bf p}_{\boldsymbol{\lambda}}(\mu)$.
 	\end{proposition}
	
For $\lambda_i$ with weight $p_i\geq 2$, there is exactly one simple object $S_i$ in $\mt_{\lambda_i}$ satisfying $\Hom_\X(\mathcal{O}, S_i)\neq 0$. Moreover, there exists a sequence of exceptional objects and epimorphims
\[S_i^{[p_i-1]}\twoheadrightarrow S_i^{[p_i-2]}\twoheadrightarrow \cdots\twoheadrightarrow S_i^{[1]}=S_i,
\]
where $S_i^{[j]}$ has length $j$ and top $S_i$.

The following is well-known (cf. \cite{GeigleLenzing}).
\begin{proposition}

Both
 \[T_{\text{can}}(\X):=\bigoplus\limits_{0\leq \vec{x}\leq \vec{c}}\mathcal{O}(\vec{x})~\text{and}~T_{\text{sq}}(\X):=\mathcal{O}\oplus \mathcal{O}(\vec{c})\oplus \bigoplus_{i=1}^t(\bigoplus_{k=1}^{p_i-1}S_i^{[p_i-k]})
 \] are tilting objects of $\coh \mathbb{X}$.
\end{proposition}

\subsection{Quivers associated with $T_{\text{can}}(\X)$ and $T_{\text{sq}}(\X)$}
Denote by $\der^b(\coh\X)$ the bounded derived category of $\coh\X$ with suspension functor $\Sigma$. Let $\tau:\der^b(\coh \X)\to \der^b(\coh\X)$ be the Auslander-Reiten (AR) translation functor, which restricts to the AR translation of $\coh\X$.
\begin{definition}
The {\it cluster category} $\mathcal{C}_\X$ associated with $\X$ is defined as the orbit category $\der^b(\coh\X)/\langle \tau^{-1}\circ \Sigma \rangle$; it has the same objects as $\der^b(\coh\X)$, morphism spaces are given by $\bigoplus_{i\in \Z}\Hom_{\der^b(\coh\X)}(X, (\tau^{-1}\circ \Sigma)^i Y)$ with obvious composition. 
\end{definition}

The cluster category $\mathcal{C}_\X$ admits a canonical triangle structure such that the projection $\pi_{\X}:\der^b(\coh\X)\to \mathcal{C}_\X$ is a triangle functor (cf. \cite{Keller}). 
The suspension functor $\Sigma$ (resp. the AR translation $\tau$) of $\der^b(\coh\X)$ induces the suspension functor (resp. the AR translation) of $\mathcal{C}_\X$, which will be denoted by $\Sigma$ (resp. $\tau$) as well.
It was shown in~\cite{Keller} that $\mathcal{C}_\X$ is a $2$-Calabi-Yau triangulated category, i.e., for any $X,Y\in \mathcal{C}_\X$, we have bifunctorially isomorphisms
\[\Hom_{\mathcal{C}_\X}(X, \Sigma^2Y)\cong \mathbb{D}\Hom_{\mathcal{C}_\X}(Y, X).
\]
By the $2$-Calabi-Yau property, we clearly have $\tau=\Sigma$ in $\mathcal{C}_\X$. 
\begin{definition}
An object $T\in \mathcal{C}_\X$ is a {\it cluster-tilting object} if $\Ext^1_{\mathcal{C}_\X}(T,T)=0$ and $\Ext^1_{\mathcal{C}_\X}(T,X)=0$ implies that $X\in \add T$, where $\add T$ is the full subcategory of $\mathcal{C}_\X$ consisting of direct summands of direct sum of finite copies of $T$.  
\end{definition}
 Since $\coh\X$ has no nonzero projective objects, the composition of the embedding of $\coh\X$ into $\der^b(\coh\X)$ with the projection functor $\pi_\X$ yields a bijection between the set of isomorphism classes of indecomposable objects of $\coh\X$ and the set of isomorphism classes of indecomposable objects of $\mathcal{C}_\X$. We may identify the objects of $\coh\X$ with the ones of $\mathcal{C}_\X$ by the bijection.
 \begin{lemma}\cite[Section 3]{BMRRT}
 An object $T\in \coh\X$ is a tilting object if and only if $T$ is a cluster-tilting object of $\mathcal{C}_\X$.
 \end{lemma}
 In particular, $T_{\text{can}}(\X)$ and $T_{\text{sq}}(\X)$ are basic cluster-tilting objects of $\mathcal{C}_\X$. We denote by $Q_{T_\text{can}(\X)}$ (resp. $Q_{T_{\text{sq}}(\X)}$) the Gabriel quiver of the endomorphism algebra $\End_{\mathcal{C}_\X}(T_{\text{can}}(\X))$ (resp. $\End_{\mathcal{C}_\X}(T_{\text{sq}}(\X))$). The quivers have been listed in Figure ~\ref{f:1} and Figure \ref{f:2} respectively. We remark that the relation of the corresponding algebra is quite complicated in general and we do not need in the sequel.

\begin{figure}[ht]

\begin{tikzpicture}[
    scale=2.5,axis/.style={ very thick, ->, >=stealth'},
    arrow/.style={->, >=stealth'},
    important line/.style={thick},
    dashed line/.style={dashed, thin},
    pile/.style={thick, ->, >=stealth', shorten <=2pt, shorten
    >=2pt},
    every node/.style={color=black}
    ]

  \draw[arrow](-0.5,0.02)--(0.5,0.02);
  \draw[arrow](-0.5,-0.02)--(0.5,-0.02);
  \draw[arrow](0.5,0.02)--(0,0.5);
  \draw[arrow](0,0.5)--(-0.5,0.02);
   \draw[arrow](0.5,-0.02)--(0,-0.25);
  \draw[arrow](0,-0.25)--(-0.5,-0.02);
  \node at (-0.55,0){\tiny{$\mathcal{O}$}};
  \node at (0.65,0){\tiny{$\mathcal{O}(\vec{c})$}};
   \node at (0.05,0.6){\tiny{$S_1^{[p_1-1]}$}};
   \node at (0.06,-0.15){\tiny{$S_2^{[p_2-1]}$}};
  \node at (0,-0.35){\tiny{$\vdots$}};
  \draw[arrow](0.5,-0.02)--(0,-0.5);
  \draw[arrow](0,-0.5)--(-0.5,-0.02);
 \node at (0.05,-0.6){\tiny{$S_t^{[p_t-1]}$}};
 
 \draw[arrow](0,0.5)--(0.5,0.5);
 \node at (0.5,0.6){\tiny{$S_1^{[p_1-2]}$}};
 \draw[dashed line](0.5,0.5)--(1,0.5);
  \node at (1,0.6){\tiny{$S_1^{[2]}$}};
  \draw[arrow](1,0.5)--(1.5,0.5);
\node at (1.5,0.6){\tiny{$S_1^{[1]}$}};
 
  \draw[arrow](0,-0.25)--(0.5,-0.25);
 \node at (0.6,-0.15){\tiny{$S_2^{[p_2-2]}$}};
 \draw[dashed line](0.5,-0.25)--(1,-0.25);
  \node at (1,-0.15){\tiny{$S_2^{[2]}$}};
  \draw[arrow](1,-0.25)--(1.5,-0.25);
\node at (1.5,-0.15){\tiny{$S_2^{[1]}$}};
 
  \draw[arrow](0,-0.5)--(0.5,-0.5);
 \node at (0.5,-0.6){\tiny{$S_t^{[p_t-2]}$}};
 \draw[dashed line](0.5,-0.5)--(1,-0.5);
  \node at (1,-0.6){\tiny{$S_t^{[2]}$}};
  \draw[arrow](1,-0.5)--(1.5,-0.5);
\node at (1.5,-0.6){\tiny{$S_t^{[1]}$}};
 
 \node at (0.5, -0.35){\tiny{$\vdots$}};
 \node at (1, -0.35){\tiny{$\vdots$}};
\end{tikzpicture}

\caption{Quiver $Q_{T_{\text{sq}}(\X)}$ with weight sequence $(p_1, \dots, p_t)$.}~\label{f:1}
\end{figure}
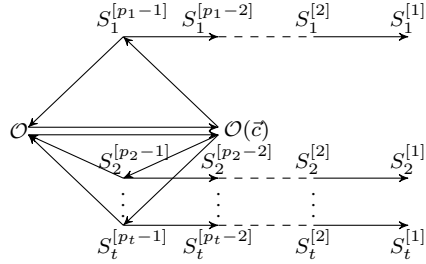

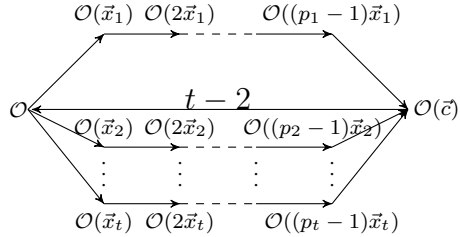
\begin{figure}[ht]
\begin{tikzpicture}[
    scale=2.5,axis/.style={ very thick, ->, >=stealth'},
    arrow/.style={ ->, >=stealth'},
    important line/.style={thick},
    dashed line/.style={dashed, thin},
    pile/.style={thick, ->, >=stealth', shorten <=2pt, shorten
    >=2pt},
    every node/.style={color=black}
    ]

  \draw[arrow](2,0)--(0.02,0);
  \node at (1,0.05){{\bf $t-2$}};
  \draw[arrow](0,0)--(0.4,0.4);
  \draw[arrow](0.4,0.4)--(0.8,0.4);
  \draw[dashed line](0.8,0.4)--(1.2, 0.4);
  \draw[arrow](1.2,0.4)--(1.6,0.4);
  \draw[arrow](1.6,0.4)--(2,0);
  
   \node at (-0.05,0){\tiny{$\mathcal{O}$}};
   \node at (2.15,0){\tiny{$\mathcal{O}(\vec{c})$}};
    \node at (0.4,0.5){\tiny{$\mathcal{O}(\vec{x}_1)$}};
   \node at (0.8,0.5){\tiny{$\mathcal{O}(2\vec{x}_1)$}};
    \node at (1.6,0.5){\tiny{$\mathcal{O}((p_1-1)\vec{x}_1)$}};
   
  \draw[arrow](0,0)--(0.4,-0.2);
  \draw[arrow](0.4,-0.2)--(0.8,-0.2);
  \draw[dashed line](0.8,-0.2)--(1.2, -0.2);
  \draw[arrow](1.2,-0.2)--(1.6,-0.2);
  \draw[arrow](1.6,-0.2)--(2,0);
  
  \node at (0.4, -0.1){\tiny{$\mathcal{O}(\vec{x}_2)$}};
  \node at (0.8, -0.1){\tiny{$\mathcal{O}(2\vec{x}_2)$}};
  \node at (1.5, -0.1){\tiny{$\mathcal{O}((p_2-1)\vec{x}_2)$}};
  \draw[arrow](0,0)--(0.4,-0.5);
  \draw[arrow](0.4,-0.5)--(0.8,-0.5);
  \draw[dashed line](0.8,-0.5)--(1.2, -0.5);
  \draw[arrow](1.2,-0.5)--(1.6,-0.5);
  \draw[arrow](1.6,-0.5)--(2,0);
  
  \node at (0.4, -0.6){\tiny{$\mathcal{O}(\vec{x}_t)$}};
  \node at (0.8, -0.6){\tiny{$\mathcal{O}(2\vec{x}_t)$}};
  \node at (1.6, -0.6){\tiny{$\mathcal{O}((p_t-1)\vec{x}_t)$}};
  \node at (0.4,-0.3){\tiny{$\vdots$}};
  \node at (0.8,-0.3){\tiny{$\vdots$}};
  \node at (1.2,-0.3){\tiny{$\vdots$}};
  \node at (1.6,-0.3){\tiny{$\vdots$}};
  
\end{tikzpicture}
\caption{Quiver $Q_{T_{\text{can}}(\X)}$ with weight sequence $(p_1,\dots, p_t)$, where the label $t-2$ means that there are $t-2$ arrows from $\mathcal{O}(\vec{c})$ to $\mathcal{O}.$ }~\label{f:2}
\end{figure}
The following is a direct consequence of \cite[Theorem 1.2]{FG}.
\begin{proposition}\label{p:mutation-equivalent}
Let $T$ be a basic cluster-tilting object of $\mathcal{C}_\X$ and $Q_T$ the Gabriel quiver of the endomorphism algebra of $T$. Then 
\begin{enumerate}
\item $Q_T$ is mutation-equivalent to $Q_{T_{\text{sq}}(\X)}$. In particular, the quiver $Q_{T_\text{can}(\X)}$ is mutation-equivalent to $Q_{T_{\text{sq}}(\X)}$.
\item $Q_T$ admits a green-to-red sequence.
\end{enumerate}

\end{proposition}

\subsection{The classification}
Denote by $p=\operatorname{lcm}(p_1,\dots, p_t)$ the least common multiple of $p_1,\dots, p_t$. The {\it genus} $g_\X$ of $\X$ is defined as
\[g_\X=1+\frac{1}{2}((t-2)p-\sum_{i=1}^t\frac{p}{p_i}).
\]
A weighted projective line of genus $g_\X<1$($g_\X=1$, resp. $g_\X>1$) is  of {\it domestic} ({\it tubular}, resp. {\it wild}) type. The domestic types are, up to permutation, $(1,p) $ with $p\geq 1$, $(p,q)$ with $p,q\geq 2$, $(2,2,n)$ with $n\geq 2$, $(2,3,3)$, $(2,3,4)$ and $(2,3,5)$, whereas the tubular types are, up to permutation, $(2,2,2,2)$, $(3,3,3)$, $(2,4,4)$ and $(2,3,6)$.
It is worth pointing out that a weighted projective line of domestic type is derived equivalent to a finite dimensional hereditary algebra of tame type.

\begin{proposition}\label{p:finite-mutation-wpl}
Let $\X$ be a weighted projective line. The quiver $Q_{T_{\text{sq}}(\X)}$ is of  finite mutation type if and only if $\X$ is of domestic type or of tubular type.
\end{proposition}
\begin{proof}
The ``if'' part follows from Lemma \ref{l:finite-mutatin-type}. More precisely, if $\X$ is of domestic type, then $\coh \X$ is derived equivalent to a finite dimensional hereditary algebra of tame type. As a consequence, the quiver $Q_{T_{\text{sq}}(\X)}$ is mutation-equivalent to an acyclic quiver of extended Dynkin type. If $\X$ is of tubular type, then the quiver $Q_{T_{\text{sq}}(\X)}$ is as in Figure \ref{f:quiver-tubular-type}.

For the ``only if" part, it suffices to prove that $Q_{T_{\text{sq}}(\X)}$ is not of finite mutation type provided that $\X$ is of wild type. According to Lemma \ref{l:non-finite-mutation-type}, it suffices to show that there is a quiver $Q$ in $\mut(Q_{T_{\text{sq}}(\X)})$ such that $Q$ admits a full subquiver $Q_{a,b,c}$ for some $2\leq c\leq b\leq a$ and $a\geq 3$.

Let $\X$ be a wild weighted projective line.
According to the classification of weighted projective lines, the quiver $Q_{T_{\text{sq}}(\X)}$ admits one of the following quivers as a subquiver
\begin{enumerate}
\item $Q_{T_{\text{sq}}(\X')}$ with weight sequence $(2,3,7)$;
\item $Q_{T_{\text{sq}}(\X')}$ with weight sequence $(2,4,5)$;
\item $Q_{T_{\text{sq}}(\X')}$ with weight sequence $(3,3,4)$;
\item $Q_{T_{\text{sq}}(\X')}$ with weight sequence $(2,2,2,3)$;
\item $Q_{T_{\text{sq}}(\X')}$ with weight sequence $(2,2,2,2,2)$.
\end{enumerate} 
Let $\mathbf{p}$ be one of the weight sequences in $(1)$-$(5)$. According to Proposition \ref{p:mutation-equivalent}, there is a  quiver $Q_{\mathbf{p}}$ in $\mut(Q_{T_{\text{sq}}(\X)})$ such that $Q_{\mathbf{p}}$ admits $Q_{T_{\text{can}}(\X')}$ as a full subquiver, where $\X'$ has the weight sequence $\mathbf{p}$. It suffices to show that there is a quiver in $\mut(Q_{T_{\text{can}}(\X')})$ which admits a subquiver $Q_{a,b,c}$ for $2\leq c\leq b\leq a$ and $3\leq a$. Let us label the vertices of $Q_{T_{\text{can}}(\X')}$ as in Figure \ref{f:2}.
For $\mathbf{p}=(2,3,7)$, let
\begin{eqnarray*}\mathfrak{i}&=&(\mathcal{O},\mathcal{O}(6\vec{x}_3),\mathcal{O}(\vec{c}),\mathcal{O}(2\vec{x}_3),\mathcal{O}(\vec{x}_3),\mathcal{O}(2\vec{x}_2),\mathcal{O}(6\vec{x}_3),\mathcal{O}(5\vec{x}_3),\\ &&\mathcal{O}(\vec{x}_2),\mathcal{O}(2\vec{x}_3),
\mathcal{O}(3\vec{x}_3),\mathcal{O}(2\vec{x}_2),\mathcal{O}(\vec{x}_3),\mathcal{O}(\vec{c})).
\end{eqnarray*} 
For $\mathbf{p}=(2,4,5)$, let
\begin{eqnarray*}
\mathfrak{i}&=&(\mathcal{O},\mathcal{O}(\vec{c}),\mathcal{O}(\vec{x}_3), \mathcal{O}(2\vec{x}_3),\mathcal{O}(3\vec{x}_3),\mathcal{O}(3\vec{x}_2),\\
&&\mathcal{O}(\vec{c}),\mathcal{O}(3\vec{x}_3),\mathcal{O}(4\vec{x}_3),\mathcal{O}(\vec{x}_2),\mathcal{O}(2\vec{x}_3)).
\end{eqnarray*}
For $\mathbf{p}=(3,3,4)$, let $\mathfrak{i}=(\mathcal{O},\mathcal{O}(\vec{x}_1),\mathcal{O}(\vec{x}_2),\mathcal{O}(\vec{x}_3),\mathcal{O}(\vec{c}),\mathcal{O})$. 
For $\mathbf{p}=(2,2,2,3)$, let $\mathfrak{i}=(\mathcal{O}(\vec{c}),\mathcal{O})$.  It is straightforward to check that $Q_{2,2,3}$ is a subquiver of $\mu_{\mathfrak{i}}(Q_{T_{\text{can}}(\X')})$ in each case. Finally, for $\mathbf{p}=(2,2,2,2,2)$, denote by $\mathfrak{i}=(\mathcal{O}(\vec{c}),\mathcal{O})$. We find that $Q_{2,3,5}$ is a subquiver of $\mu_{\mathcal{O}}\mu_{\mathcal{O}(\vec{c})}(Q_{T_{\text{can}}(\X')})$ in this case. This completes the proof.

\end{proof}

\section{The existence and non-existence of maximal green sequence}\label{s:proof-main-result}
This section is devoted to proving the main result of this note. We begin with the hyperbolic case. Recall that a weighted projective line $\X$ with weight sequence $(p_1,\dots, p_t)$ is of {\it hyperbolic type} if $p_1=p_2=\cdots=p_t=2$.

Let  $\X$ be of hyperbolic type.  We denote by $Q_t$ the quiver $Q_{T_{\text{sq}}(\X)}$ in this case and relabel the vertices of $Q_t$ as in Figure~\ref{f:3}. We will always identify $Q_{t}$ with a full subquiver of $Q_{t+1}$ such that the vertex $t+1$ is the unique vertex which does not belong to $Q_t$.
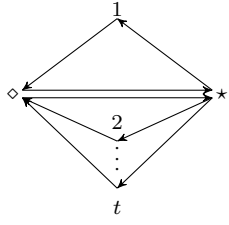
\begin{figure}
\begin{tikzpicture}[
    scale=2.5,axis/.style={ very thick, ->, >=stealth'},
    arrow/.style={ ->, >=stealth'},
    important line/.style={thick},
    dashed line/.style={dashed, thin},
    pile/.style={thick, ->, >=stealth', shorten <=2pt, shorten
    >=2pt},
    every node/.style={color=black}
    ]

\draw[arrow] (-0.5, 0.02)--(0.5, 0.02);
\draw[arrow] (-0.5, -0.02)--(0.5, -0.02);
\draw[arrow] (0.5, 0.02)--(0, 0.4);
\draw[arrow] (0, 0.4)--(-0.5, 0.02);
\node at (-0.55, 0){\tiny{$\diamond$}};
\node at (0.55,0){\tiny{$\star$}};
\node at (0,0.45){\tiny{$1$}};
\draw[arrow] (0.5, -0.02)--(0, -0.25);
\draw[arrow] (0, -0.25)--(-0.5, -0.02);
\node at (0,-0.15){\tiny{$2$}};
\draw[arrow] (0.5, -0.02)--(0, -0.5);
\draw[arrow] (0, -0.5)--(-0.5, -0.02);
\node at (0, -0.6){\tiny{$t$}};
\node at (0, -0.3){\tiny{$\vdots$}};

\end{tikzpicture}
\caption{Quiver $Q_{t}=Q_{T_{\text{sq}}(\X)}$ with weight sequence $(2,2,\dots, 2)$.}~\label{f:3}
\end{figure}
\begin{lemma}~\label{l:-mgs-induction}
Let ${\bf \mathfrak{i}_{t}}$ be a maximal green sequence of $Q_{t}$. Denote by $\bullet\rightrightarrows \circ$ the unique multiple arrows in $\mu_{\bf\mathfrak{i}_t}(Q_{t})$. If ${\bf \mathfrak{i}_{t+1}}:=({\bf \mathfrak{i}_{t}},\tiny{t+1},\circ,\bullet)$ is a maximal green sequence of $Q_{t+1}$, then $\circ\rightrightarrows \tiny{t+1}$ is the unique multiple arrows in $\mu_{\bf\mathfrak{i}_{t+1}}(Q_{t+1})$ and ${\bf \mathfrak{i}_{t+2}}:=( {\bf\mathfrak{i}_{t+1}}, t+2,t+1,\circ)$ is a maximal green sequence of $Q_{t+2}$.
\end{lemma}
\begin{proof}
We apply $\mu_{\bf\mathfrak{i}_{t}}$ to the quiver $Q_{t+2}$.  Since  ${\bf \mathfrak{i}_t}$ is a sequence  of vertices of $Q_t$, $\mu_{\bf \mathfrak{i}_t}(Q_{t})$ is a full subquiver of $\mu_{\bf \mathfrak{i}_t}(Q_{t+2})$.
In particular, the vertex set of $\mu_{\bf \mathfrak{i}_t}(Q_{t})$ is a subset of the vertex set of $\mu_{\bf \mathfrak{i}_t}(Q_{t+2})$. Since $\mu_{\bf \mathfrak{i}_t}(Q_{t})\cong Q_t$,  we will denote the vertex set of $\mu_{\bf \mathfrak{i}_t}(Q_{t})$ by $\{\bullet, \circ, 1,\dots, t\}$ and the vertex set of $\mu_{\bf \mathfrak{i}_t}(Q_{t+2})$ by $\{\bullet, \circ, 1,\dots, t,t+1,t+2\}$.

Let $\hat{B}=(b_{ij})\in M_{2(t+4) }(\Z)$ be the skew-symmetric matrix associated to the framed quiver $\hat{Q}_{t+2}$ and $\hat{B}^{\circ}$ the submatrix of $\hat{B}$ consisting of the first $t+4$ columns. We index the columns of $\hat{B}^{\circ}$ by $\bullet, \circ, 1,\dots, t+2$.

\noindent{\bf Claim $1$:} The principal part of $\mu_{\bf\mathfrak{i}_t}(\hat{B}^\circ)$ is 
 \[\tiny\begin{bmatrix}
 0&-2&1&\cdots&1&1&1\\
 2&0&-1&\cdots&-1&-1&-1\\
 -1&1&0&\cdots&0&0&0\\
 &\vdots&&\vdots&&\vdots&\\
 -1&1&0&\cdots&0&0&0\\
 -1&1&0&\cdots&0&0&0\\
 -1&1&0&\cdots&0&0&0
 \end{bmatrix}\in M_{t+4}(\Z).
 \]
The proof of this claim will be separated into three steps.
Here we work on the quivers.

\noindent{\bf Step $1$:}
{\it The full subquiver of $\mu_{\bf\mathfrak{i}_{t}}(Q_{t+2})$ consisting of vertices $\bullet, \circ, t+1$ and  $t+2$ has the following form:}
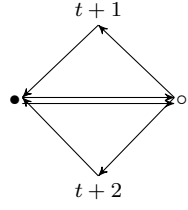
\begin{figure}[h]
\begin{tikzpicture}[
    scale=2,axis/.style={ very thick, ->, >=stealth'},
    arrow/.style={ ->, >=stealth'},
    important line/.style={thick},
    dashed line/.style={dashed, thin},
    pile/.style={thick, ->, >=stealth', shorten <=2pt, shorten
    >=2pt},
    every node/.style={color=black}
    ]
    
    \draw[arrow](0,0.02)--(1, 0.02);
    \draw[arrow](0,-0.02)--(1, -0.02);
    \node at(-0.05,0){\tiny{$\bullet$}};
     \node at(1.05,0){\tiny{$\circ$}};
   
 \draw[arrow](1,0.02)--(0.5, 0.5);
  \draw[arrow](0.5,0.5)--(0, 0.02);
 
   \node at(0.5,0.6){\tiny{$t+1$}};
   \node at(0.5,-0.6){\tiny{$t+2$}};
   
   \draw[arrow](1,0.02)--(0.5, -0.5);
   \draw[arrow](0.5,-0.5)--(0, 0.02);

    \end{tikzpicture}
\end{figure}

Suppose that there are $a$ arrows from vertex $t+1$ to vertex $\bullet$ and $b$ arrows from vertex $\circ$ to vertex $t+1$ and denote by $Q(a,b)$ the full subquiver consisting of vertices $\bullet, \circ$ and $t+1$. Since there is a symmetry between $t+1$ and $t+2$ in $Q_{t+2}$, it suffices to prove that $a=b=1$.

Denote by $B=\begin{bmatrix}0&-2&a\\ 2&0&-b\\ -a &b&0\end{bmatrix}$ the associated skew-symmetric matrix of $Q(a,b)$.  Denote by 
\[c_{12}=-2+\sgn(a)[ab]_+, c_{13}=a-\sgn(c_{12})[bc_{12}]_+, c_{23}=-b+\sgn(c_{12})[-c_{12}c_{13}]_+.
\]

By Fomin-Zelevinsky's matrix mutation formula, we obtain
\[C:=\mu_{\bullet}(\mu_{\circ}(\mu_{t+1}(B)))=\begin{bmatrix} 0&c_{12}& c_{13}\\ -c_{12}&0&c_{23}\\ -c_{13}&-c_{23}&0
\end{bmatrix}.
\]
Note that the associated skew-symmetric matrix of $\mu_\bullet(\mu_\circ(\mu_{t+1}(Q(a,b))))$ is the matrix $C$. By the assumption that ${\bf \mathfrak{i}_{t+1}}$ is a maximal green sequence of $Q_{t+1}$, we have $\mu_{\bf \mathfrak{i}_{t+1}}(Q_{t+1})\cong Q_{t+1}$. In particular, the quiver $\mu_\bullet(\mu_\circ(\mu_{t+1}(Q(a,b))))$ is a full subquiver of $Q_{t+1}$ via the isomorphism $\mu_{\bf \mathfrak{i}_{t+1}}(Q_{t+1})\cong Q_{t+1}$. The remaining proof is a discussion of the values of $a$ and $b$, from which we can deduce that $a=1=b$. We will denote by $Q(C)$ the associated quiver of the skew-symmetric matrix $C$.

\noindent{\it Case $1$: $a<0$, $b>0$.}  A direct computation shows that $c_{12}=-2$, $c_{13}=a$ and $c_{23}=-b$. Consequently, the associate quiver $Q(C)$ is not a full subquiver of $Q_{t+1}$.

\noindent{\it Case $2$: $a>0, b<0$.} We have $c_{12}=-2$, $c_{13}=a-2b\geq 3$, which implies that the associated quiver $Q(C)$ is not a full subquiver of $Q_{t+1}$.

\noindent{\it Case $3$: $a\leq 0, b\leq 0$.}  We have $c_{12}=-2-ab\leq -2$. Since $Q(C)$ is a full subquiver of $Q_{t+1}$, we have $c_{12}=-2$. Hence $ab=0$, i.e., $a=0$ or $b=0$. In each case, one can show that $Q(C)$ is not a full subquiver of $Q_{t+1}$.

\noindent{\it Case $4$: $a\geq 0$, $b\geq 0$.} Similar to the Case $3$, we obtain $-2\leq c_{12}=-2+ab\leq 2$. In particular, $0\leq ab\leq 4$. A direct computation shows that $a=1=b$ is the unique value such that $Q(C)$ is a full subquiver of $Q_{t+1}$.  This completes the proof for the statement in {\bf Step $1$}.

As a direct consequence of the statement of {\bf Step $1$},  the quiver $Q(C)$ has the form as in Figure~\ref{f:quiver-c}.  Consequently, $\circ\rightrightarrows t+1$ is the unique multiple arrows in $\mu_{\bf\mathfrak{i}_{t+1}}(Q_{t+1})$.

\begin{figure}[h]
\begin{tikzpicture}[
    scale=2,axis/.style={ very thick, ->, >=stealth'},
    arrow/.style={ ->, >=stealth'},
    important line/.style={thick},
    dashed line/.style={dashed, thin},
    pile/.style={thick, ->, >=stealth', shorten <=2pt, shorten
    >=2pt},
    every node/.style={color=black}
    ]
    
    \draw[arrow](0,0.02)--(1, 0.02);
    \draw[arrow](0,-0.02)--(1, -0.02);
    \node at(-0.05,0){\tiny{$\circ$}};
     \node at(1.25,0){\tiny{$t+1$}.};

 \draw[arrow](1,0.02)--(0.5, 0.5);
  \draw[arrow](0.5,0.5)--(0, 0.02);

   \node at(0.5,0.55){\tiny{$\bullet$}};
    \end{tikzpicture}
    \caption{Quiver $Q(C)$.}~\label{f:quiver-c}
\end{figure}
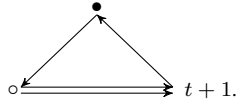

\noindent{\bf Step $2$:} {\it  There are no arrows between vertex $k\in \{1,\dots, t\}$ and vertex $t+1$ in the quiver $\mu_{\mathfrak{i}_t}(Q_{t+2})$.}

Without loss of generality, we may assume that there are $a$ arrows from vertex $i$ to vertex $t+1$ and  we consider the full subquiver consisting of vertices $\bullet, \circ, t+1, i$:
\begin{figure}[h]
	\begin{tikzpicture}[
	scale=2,axis/.style={ very thick, ->, >=stealth'},
	arrow/.style={ ->, >=stealth'},
	important line/.style={thick},
	dashed line/.style={dashed, thin},
	pile/.style={thick, ->, >=stealth', shorten <=2pt, shorten
		>=2pt},
	every node/.style={color=black}
	]
	
	\draw[arrow](0,0.02)--(1, 0.02);
	\draw[arrow](0,-0.02)--(1, -0.02);
	\node at(-0.05,0){\tiny{$\bullet$}};
	\node at(1.05,0){\tiny{$\circ$}.};
	\node at(0.5, -0.5){\tiny{$i$}};
	
	\draw[arrow](1,0.02)--(0.5, 0.5);
	\draw[arrow](0.5,0.5)--(0, 0.02);
	\draw[arrow](1.,0)--(0.5,-0.45);
	\draw[arrow](0.5,-0.45)--(0,-0.02);
	\draw[arrow](0.5,-0.45)--(0.5,0.45);
	
	\node at(0.55, 0.25){\tiny{$a$}};
	
	\node at(0.5,0.55){\tiny{$t+1$}};
	\end{tikzpicture}
\end{figure}

By applying the mutation sequence $t+1, \circ, \bullet$ to the above quiver, we obtain 
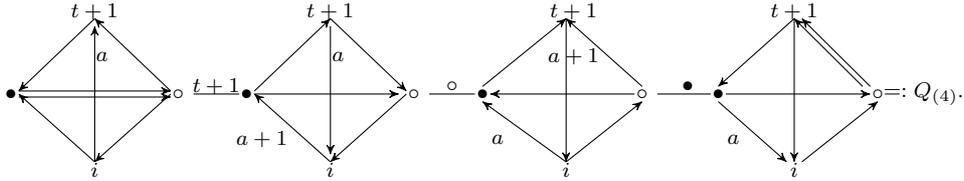
\begin{figure}[h]
	\begin{tikzpicture}[
	scale=2,axis/.style={ very thick, ->, >=stealth'},
	arrow/.style={ ->, >=stealth'},
	important line/.style={thick},
	dashed line/.style={dashed, thin},
	pile/.style={thick, ->, >=stealth', shorten <=2pt, shorten
		>=2pt},
	every node/.style={color=black}
	]
	
	\draw[arrow](-1,0.02)--(0, 0.02);
	\draw[arrow](-1,-0.02)--(0, -0.02);
	\node at(-1.05,0){\tiny{$\bullet$}};
	\node at(0.05,0){\tiny{$\circ$}};
	\node at(-0.5, -0.5){\tiny{$i$}};
	
	\draw[arrow](0,0.02)--(-0.5, 0.5);
	\draw[arrow](-0.5,0.5)--(-1, 0.02);
	\draw[arrow](0,0)--(-0.5,-0.45);
	\draw[arrow](-0.5,-0.45)--(-1,0);
	\draw[arrow](-0.5,-0.45)--(-0.5,0.45);
	
	\node at(-0.45, 0.25){\tiny{$a$}};
	
	\node at(-0.5,0.55){\tiny{$t+1$}};
	
	\draw[-](0.15, 0)--(0.45,0);
	\node at (0.3, 0.05){\tiny{$t+1$}};
	~~~~~~~~~~~~~~~~~~~~~~~~~~~~~~~~~~~~~~~~~~~~~~~
	\draw[arrow](0.55,0)--(1.5, 0);
	
	\node at(0.5,0){\tiny{$\bullet$}};
	\node at(1.6,0){\tiny{$\circ$}};
	\node at(1.05, -0.5){\tiny{$i$}};
	
	\draw[arrow](1.05, 0.5)--(1.55,0.02);
	\draw[arrow](0.55, 0.02)--(1.05,0.5);
	\draw[arrow](1.55,0)--(1.05,-0.45);
	\draw[arrow](1.05,-0.45)--(0.55,0);
	\draw[arrow](1.05,0.45)--(1.05,-0.4);
	
	\node at(1.1, 0.25){\tiny{$a$}};
	\node at(0.6,-0.3){\tiny{$a+1$}};
	
	\node at(1.05,0.55){\tiny{$t+1$}};
	
	\draw[-](1.7, 0)--(2,0);
	\node at (1.85, 0.05){\tiny{$\circ$}};
	~~~~~
	\draw[arrow](3.05, 0)--(2.1,0);
	
	\node at(2.05,0){\tiny{$\bullet$}};
	\node at(3.1,0){\tiny{$\circ$}};
	\node at(2.6, -0.5){\tiny{$i$}};

	\draw[arrow](3.1,0.05)--(2.6,0.5);
	\draw[arrow](2.05,0.05)--(2.6,0.5);
	\draw[arrow](2.6,0.5)--(2.6,-0.45);
	\draw[arrow](2.6,-0.45)--(2.05,-0.05);
	\draw[arrow](2.6,-0.45)--(3.1,-0.05);
	
	\node at(2.65, 0.25){\tiny{$a+1$}};
	\node at(2.15,-0.3){\tiny{$a$}};
	
	\node at(2.65,0.55){\tiny{$t+1$}};
	
	\draw[-](3.2, 0)--(3.55,0);
	\node at (3.4, 0.05){\tiny{$\bullet$}};
	~~~~~~~~~~~~~~~~~~~~
	\draw[arrow](3.65,0)--(4.6, 0);
	
	\node at(3.6,0){\tiny{$\bullet$}};
	\node at(4.65,0){\tiny{$\circ$}};
	\node at(4.1, -0.5){\tiny{$i$}};
	\node at (4.95,0){\tiny{$=: Q_{(4)}$.}};
	
	\draw[arrow](4.55,0.05)--(4.1, 0.5);
	\draw[arrow](4.6,0.05)--(4.15,0.5);
	\draw[arrow](4.1,0.5)--(3.6,0.05);
	\draw[arrow](4.1,0.5)--(4.1,-0.45);
	\draw[arrow](4.15,-0.45)--(4.65,-0.05);
	\draw[arrow](3.6,-0.05)--(4.05,-0.45);
	
	\node at(3.7,-0.3){\tiny{$a$}};
	
	\node at(4.1,0.55){\tiny{$t+1$}};
	\end{tikzpicture}
\end{figure}

By the assumption that ${\bf \mathfrak{i}_{t+1}}$ is a maximal green sequence of $Q_{t+1}$, we know that the quiver $Q_{(4)}$ is a full subquiver of $\mu_{\bf \mathfrak{i}_{t+1}}(Q_{t+1})$ and we conclude that $a=0$. This completes the proof of the statement in {\bf Step $2$}.

\noindent{\bf Step $3$:} {\it There are no arrows between vertex $k\in \{1,\dots, t+1\}$ and vertex $t+2$ in the quiver $\mu_{\mathfrak{i}_t}(Q_{t+2})$.}

Note that there is a symmetry between vertex $t+1$ and $t+2$ in the quiver $Q_{t+2}$ and the mutation sequence ${\bf \mathfrak{i}_t}$ does not involve the vertices $t+1$ and $t+2$. Then the statement of {\bf Step $3$} follows the statement of {\bf Step $2$} directly.

Now {\bf Claim $1$} is a direct consequence of the statements in {\bf Step $1, 2, 3$.}

 \noindent{\bf Claim $2$:} {\it Up to a permutation of  the rows associated to $\bullet^*, \circ^*, 1^*,\dots, t^*$, the coefficient part $C_{\bf \mathfrak{i}_t}$ of $\mu_{\bf\mathfrak{i}_t}(\hat{B}^\circ)$  has the following form:}
 \[\begin{bmatrix}
 -I_{t+2}&X\\ 0& I_2
 \end{bmatrix}
 \]
 {\it where  $X\in M_{(t+2)\times 2}(\Z)$ with non-negative entries. } 

We fix a quiver pattern of the framed quiver $\hat{Q}_{t+2}$ of $Q_{t+2}$ by assigning $\hat{Q}_{t+2}$ to the root vertex $t_0\in \mathbb{T}_{t+4}$.
Each sequence $\mathfrak{i}$ of vertices of $Q_{t+2}$ induces a path of $\mathbb{T}_{t+4}$ with starting point $t_0$. We denote by the ending point $s_{\mathfrak{i}}$ and denote by $G_{\mathfrak{i}}:=G_{s_{\mathfrak{i}}}$ (resp. $C_{\mathfrak{i}}$) the $G$-matrix (resp. $C$-matrix) at $s_{\mathfrak{i}}$.
 
 Since the sequence ${\bf \mathfrak{i}_t}$ does not involves the vertices $t+1$ and $t+2$. It follows that 
 $G_{\bf \mathfrak{i}_t}=\begin{bmatrix} A&0\\ Y&I_2
 \end{bmatrix}$,
 where $A\in M_{t+2}(\Z)$ is invertible and $Y\in M_{2\times (t+2)}(\Z)$ with non negative entries.
 By the tropical dualities between $G$-matrices and $C$-matrices (\ref{t:tropical-duality}), we have 
 \[C_{\bf \mathfrak{i}_t}=G_{\bf\mathfrak{i}_t}^{-T}=\begin{bmatrix}
 A^{-T}&-A^{-T}Y^T\\ 0&I_2
 \end{bmatrix}.
 \]
 Since ${\bf \mathfrak{i}_t}$ is a maximal green sequence, it follows that $A^{-T}$ is a permutation of $-I_{t+2}$. On the other hand,  the entries of $-A^{-T}Y^T$ are non negative by the sign-coherence of $c$-vectors. This finishes the proof of {\bf Claim $2$}.

According to {\bf Claim $1$} and {\bf Claim $2$}, up to permutation of indices, we may assume that
\[\tiny \mu_{\bf \mathfrak{i}_t}(\hat{B}^\circ)=\begin{blockarray}{cccccccc}
\text{$\bullet$}&\text{$\circ$}&\text{\color{blue}$1$}&\text{\color{blue}$\cdots$}&\text{\color{blue}$t$}&\text{\color{blue}$t+1$}&\text{\color{blue}$t+2$}&\\
\begin{block}{[ccccccc]c}
0&-2&1&\cdots&1&1&1& \text{$\bullet$}\\
2&0&-1&\cdots&-1&-1&-1& \text{$\circ$}\\
-1&1&0&\cdots&0&0&0& \text{\color{blue}$1$}\\
&\vdots&&\vdots&&\vdots& &\text{\color{blue}$\vdots$}\\
-1&1&0&\cdots&0&0&0& \text{\color{blue}$t$}\\
-1&1&0&\cdots&0&0&0& \text{\color{blue}$t+1$}\\
-1&1&0&\cdots&0&0&0&\text{\color{blue}$t+2$}\\  
-1&0&0&\cdots&0&a_{\bullet}&a_{\bullet}&\\
0&-1&0&\cdots&0&a_{\circ}&a_{\circ}&\\
0&0&-1&\cdots&0&a_1&a_1&\\
&\vdots&&\vdots&&\vdots&&\\
0&0&0&\cdots&-1&a_t&a_t&\\
0&0&0&\cdots&0&1&0&\\
0&0&0&\cdots&0&0&1&\\
\end{block}
\end{blockarray}
\]
where $\bullet, \circ, {\color{blue}1},\cdots, {\color{blue} t+2}$ are (relabelled) vertices of $\mu_{\bf\mathfrak{i}_t}(Q_{t+2})$, $a_{\bullet}, a_{\circ}, a_1,\ldots, a_t$ are non negative integers.

By Fomin-Zelevinsky's mutation rule, we obtain $\mu_{\bf \mathfrak{i}_{t+1}}(\hat{B}^\circ)$ as in Figure \ref{m:t+1}.
\begin{figure}

\begin{eqnarray*} \mu_{\bf \mathfrak{i}_{t+1}}(\hat{B}^\circ)&=&\mu_{\bullet}\circ \mu_{\circ}\circ \mu_{t+1}(\mu_{\bf \mathfrak{i}_t}(\hat{B}^{\circ}))\\
&=& \tiny\begin{blockarray}{cccccccc}
\text{$\bullet$}&\text{$\circ$}&\text{\color{blue}$1$}&\text{\color{blue}$\cdots$}&\text{\color{blue}$t$}&\text{\color{blue}$t+1$}&\text{\color{blue}$t+2$}&\\
\begin{block}{[ccccccc]c}
0&-1&0&\cdots&0&1&0& \text{$\bullet$}\\
1&0&1&\cdots&1&-2&1& \text{$\circ$}\\
0&-1&0&\cdots&0&1&0& \text{\color{blue}$1$}\\
&\vdots&&\vdots&&\vdots& &\text{\color{blue}$\vdots$}\\
0&-1&0&\cdots&0&1&0& \text{\color{blue}$t$}\\
-1&2&-1&\cdots&-1&0&-1& \text{\color{blue}$t+1$}\\
0&-1&0&\cdots&0&1&0&\text{\color{blue}$t+2$}\\  
1-a_{\bullet}&-1&0&\cdots&0&0&a_{\bullet}&\\
1-a_{\circ}&0&0&\cdots&0&-1&a_{\circ}&\\
-a_1&0&-1&\cdots&0&0&a_1&\\
&\vdots&&\vdots&&\vdots&&\\
-a_t&0&0&\cdots&-1&0&a_t&\\
-1&0&0&\cdots&0&0&0&\text{\color{blue}$t+1^*$}\\
0&0&0&\cdots&0&0&1&\text{\color{blue}$t+2^*$}\\
\end{block}
\end{blockarray}
\end{eqnarray*}
\caption{The matrix $\mu_{\bf \mathfrak{i}_{t+1}}(\hat{B}^\circ)$.}\label{m:t+1}
\end{figure}
 Note that ${\bf \mathfrak{i}_{t+1}}=( {\bf \mathfrak{i}_t},t+1, \circ,\bullet)$ is a maximal green sequence of $Q_{t+1}$. It follows that the submatirx formed by the first $t+3$ row indices and the first $t+3$ column indices of the coefficient part of $\mu_{\bf \mathfrak{i}_{t+1}}(\hat{B}^\circ)$ is a permutation of $-I_{t+1}$.
Hence we have 
\[a_\bullet=1, a_\circ=1, a_1=0,\ldots, a_t=0.
\]
Finally, we apply the mutation sequence $\circ, t+1, t+2$ to the matrix $\mu_{\bf \mathfrak{i}_{t+1}}(\hat{B}^\circ)$, we compute the matrix $\mu_{\bf \mathfrak{i}_{t+2}}(\hat{B}^\circ)$  as in Figure \ref{m:t+2}.
\begin{figure}
\[\xymatrix{{\tiny \mu_{\bf \mathfrak{i}_{t+1}}(\hat{B}^\circ)= \begin{blockarray}{cccccccc}
\text{$\bullet$}&\text{$\circ$}&\text{\color{blue}$1$}&\text{\color{blue}$\cdots$}&\text{\color{blue}$t$}&\text{\color{blue}$t+1$}&\text{\color{blue}$t+2$}&\\
\begin{block}{[ccccccc]c}
0&-1&0&\cdots&0&1&0& \text{$\bullet$}\\
1&0&1&\cdots&1&-2&1& \text{$\circ$}\\
0&-1&0&\cdots&0&1&0& \text{\color{blue}$1$}\\
&\vdots&&\vdots&&\vdots& &\text{\color{blue}$\vdots$}\\
0&-1&0&\cdots&0&1&0& \text{\color{blue}$t$}\\
-1&2&-1&\cdots&-1&0&-1& \text{\color{blue}$t+1$}\\
0&-1&0&\cdots&0&1&0&\text{\color{blue}$t+2$}\\  
0&-1&0&\cdots&0&0&1&\\
0&0&0&\cdots&0&-1&1&\\
0&0&-1&\cdots&0&0&0&\\
&\vdots&&\vdots&&\vdots&&\\
0&0&0&\cdots&-1&0&0&\\
-1&0&0&\cdots&0&0&0&\text{\color{blue}$t+1^*$}\\
0&0&0&\cdots&0&0&1&\text{\color{blue}$t+2^*$}\\
\end{block}
\end{blockarray}}\ar[r]^{\quad\ \ \ ~~~t+2}&{\tiny\begin{blockarray}{cccccccc}
\text{$\bullet$}&\text{$\circ$}&\text{\color{blue}$1$}&\text{\color{blue}$\cdots$}&\text{\color{blue}$t$}&\text{\color{blue}$t+1$}&\text{\color{blue}$t+2$}&\\
\begin{block}{[ccccccc]c}
0&-1&0&\cdots&0&1&0& \text{$\bullet$}\\
1&0&1&\cdots&1&-1&-1& \text{$\circ$}\\
0&-1&0&\cdots&0&1&0& \text{\color{blue}$1$}\\
&\vdots&&\vdots&&\vdots& &\text{\color{blue}$\vdots$}\\
0&-1&0&\cdots&0&1&0& \text{\color{blue}$t$}\\
-1&1&-1&\cdots&-1&0&1& \text{\color{blue}$t+1$}\\
0&1&0&\cdots&0&-1&0&\text{\color{blue}$t+2$}\\  
0&-1&0&\cdots&0&1&-1&\\
0&0&0&\cdots&0&0&-1&\\
0&0&-1&\cdots&0&0&0&\\
&\vdots&&\vdots&&\vdots&&\\
0&0&0&\cdots&-1&0&0&\\
-1&0&0&\cdots&0&0&0&\text{\color{blue}$t+1^*$}\\
0&0&0&\cdots&0&1&-1&\text{\color{blue}$t+2^*$}\\
\end{block}
\end{blockarray}}\ar[d]^{t+1}\\ \mu_{\bf \mathfrak{i}_{t+2}}(\hat{B}^{\circ})={\tiny\begin{blockarray}{cccccccc}
\text{$\bullet$}&\text{$\circ$}&\text{\color{blue}$1$}&\text{\color{blue}$\cdots$}&\text{\color{blue}$t$}&\text{\color{blue}$t+1$}&\text{\color{blue}$t+2$}&\\
\begin{block}{[ccccccc]c}
0&0&0&\cdots&0&-1&1& \text{$\bullet$}\\
0&0&0&\cdots&0&-1&1& \text{$\circ$}\\
0&0&0&\cdots&0&-1&1& \text{\color{blue}$1$}\\
&\vdots&&\vdots&&\vdots& &\text{\color{blue}$\vdots$}\\
0&0&0&\cdots&0&-1&1& \text{\color{blue}$t$}\\
1&1&1&\cdots&1&0&-2& \text{\color{blue}$t+1$}\\
-1&-1&-1&\cdots&-1&2&0&\text{\color{blue}$t+2$}\\  
0&0&0&\cdots&0&-1&0&\\
0&0&0&\cdots&0&0&-1&\\
0&0&-1&\cdots&0&0&0&\\
&\vdots&&\vdots&&\vdots&&\\
0&0&0&\cdots&-1&0&0&\\
-1&0&0&\cdots&0&0&0&\text{\color{blue}$t+1^*$}\\
0&-1&0&\cdots&0&0&0&\text{\color{blue}$t+2^*$}\\
\end{block}
\end{blockarray}} &{\tiny\begin{blockarray}{cccccccc}
\text{$\bullet$}&\text{$\circ$}&\text{\color{blue}$1$}&\text{\color{blue}$\cdots$}&\text{\color{blue}$t$}&\text{\color{blue}$t+1$}&\text{\color{blue}$t+2$}&\\
\begin{block}{[ccccccc]c}
0&0&0&\cdots&0&-1&1& \text{$\bullet$}\\
0&0&0&\cdots&0&1&-1& \text{$\circ$}\\
0&0&0&\cdots&0&-1&1& \text{\color{blue}$1$}\\
&\vdots&&\vdots&&\vdots& &\text{\color{blue}$\vdots$}\\
0&0&0&\cdots&0&-1&1& \text{\color{blue}$t$}\\
1&-1&1&\cdots&1&0&-1& \text{\color{blue}$t+1$}\\
-1&1&-1&\cdots&-1&1&0&\text{\color{blue}$t+2$}\\  
0&0&0&\cdots&0&-1&0&\\
0&0&0&\cdots&0&0&-1&\\
0&0&-1&\cdots&0&0&0&\\
&\vdots&&\vdots&&\vdots&&\\
0&0&0&\cdots&-1&0&0&\\
-1&0&0&\cdots&0&0&0&\text{\color{blue}$t+1^*$}\\
0&1&0&\cdots&0&-1&0&\text{\color{blue}$t+2^*$}\\
\end{block}
\end{blockarray}}\ar[l]_{\ \ \ \ \circ}
}
\]
\caption{The matrix  $\mu_{\bf \mathfrak{i}_{t+2}}(\hat{B}^\circ)$.}\label{m:t+2}
\end{figure}
Note that the coefficient part of the matrix $\mu_{\bf \mathfrak{i}_{t+2}}(\hat{B}^\circ)$ is a permutation of $-I_{t+4}$. According to Remark~\ref{r:mgs}, we conclude that ${\bf \mathfrak{i}_{t+2}}$ is  a maximal green sequence of $Q_{t+2}$. This completes the proof of the lemma.

\end{proof}

\begin{proposition}\label{p:existenc-mgs}
Assume that $t\geq 3$.
The quiver $Q_{T_{\text{sq}}(\X)}$ admits a maximal green sequence.
\end{proposition}
\begin{proof}
According to Lemma~\ref{l:CL-lemma}, it suffices to show that the quiver $Q_t$ admits a maximal green sequence for $t\geq 3$. We label the vertices of $Q_t$ as in Figure~\ref{f:3}. It is straightforward to check that $\mathfrak{i}_3=(\diamond, 1,2,\diamond, \star, 3,2,1,\star, \diamond)$ is a maximal green sequence for $Q_3$ and $\star\rightrightarrows 3$ is the unique multiple arrows of $\mu_{\mathfrak{i}_3}(Q_3)$. Furthermore, $\mathfrak{i}_4:=(\mathfrak{i}_3, 4,3,\star)$ is a maximal green sequence of $Q_4$. Now the result follows from Lemma~\ref{l:-mgs-induction}.
\end{proof}

\begin{theorem}\label{t:main-result}
Let $\X$ be a weighted projective line.
\begin{itemize}
\item[(1)] There is a quiver $Q'$ in $\mut(Q_{T_{\text{can}}(\X)})$ such that $Q'$ admits a maximal green sequence.
\item[(2)] There is a quiver $Q''$ in $\mut(Q_{T_{\text{can}}(\X)})$ such that $Q''$ does not admit a maximal green sequence if and only if $\X$ is of wild type.
\end{itemize}
\end{theorem}
\begin{proof}
	If $t=2$, then $Q_{T_{\text{can}}(\X)}$ is an acyclic quiver. Hence $Q_{T_{\text{can}}(\X)}$ admits a maximal green sequce (cf. \cite{BDP} for instance). Now assume that $t\geq 3$.
According to Proposition \ref{p:mutation-equivalent}, we know that $Q_{T_{\text{sq}}(\mathbb{X})}$ belongs to $\mut(Q_{T_{\text{can}}(\X)})$. Consequently, the first statement follows from Proposition~\ref{p:existenc-mgs} directly. 

The ``only if" part of $(2)$ follows from the main result of \cite{Mills}. Namely, let us assume that $\X$ is not of wild type, then $Q_{T_{\text{sq}}(\X)}$ is of finite mutation type by Proposition \ref{p:finite-mutation-wpl}. According to the main result of \cite{Mills}, each quiver in $\mut(Q_{T_{\text{sq}}(\X)})$ admits a maximal green sequence.  To prove the ``if" part, we use Lemma \ref{l:muller-existence}.
Let $\X$ be a weighted projective line of wild type. Similar to the proof of the ``only if" of Proposition \ref{p:finite-mutation-wpl}, we conclude that there is a quiver $Q$ in $\mut(Q_{T_{\text{can}}(\X)})$ such that $Q_{a,b,c}$ is a full subquiver of $Q$, where $2\leq c\leq b\leq a$ and $3\leq a$. Consequently, $Q$ does not admit a maximal green sequence by Lemma \ref{l:muller-existence}.
\end{proof}


\def\cprime{$'$} \def\cprime{$'$}
\providecommand{\bysame}{\leavevmode\hbox to3em{\hrulefill}\thinspace}
\providecommand{\MR}{\relax\ifhmode\unskip\space\fi MR }
\providecommand{\MRhref}[2]{%
  \href{http://www.ams.org/mathscinet-getitem?mr=#1}{#2}
}
\providecommand{\href}[2]{#2}
 
\end{document}